\documentclass[11pt,twoside]{article}
\usepackage[margin=1.2in]{geometry}

\usepackage{latexsym,amsfonts,amsmath,amsthm,amssymb,makeidx}
\usepackage[title]{appendix}
\usepackage{CJK,CJKnumb,CJKulem,times,dsfont,ifthen,mathrsfs,latexsym,amsfonts, color}
\usepackage{amsmath,amsthm,makeidx,fontenc,amssymb,bm,graphicx,psfrag,listings, curves,extarrows}
\usepackage{cite}
\usepackage[colorlinks,citecolor=black,linkcolor=black]{hyperref}
\usepackage{esint}

\let\oldbibliography\thebibliography
\renewcommand{\thebibliography}[1]{%
\oldbibliography{#1}%
\setlength{\itemsep}{0pt}%
}

\makeindex

\newtheorem*{pro*}{Proposition~$3.7^{(')}$}

\newtheorem{Remark}{Remark}[section]
\newtheorem{thm}{Theorem}[section]
\newtheorem{pro}{Proposition}[section]

\newtheorem{lem}{Lemma}[section]

\numberwithin{equation}{section}

\begin{document}
\title{ {\bf Super polyharmonic property and asymptotic behavior of solutions to the higher order Hardy-H\'enon equation near isolated singularities} }
\date{\today}
\author{\\ Xia Huang,
\;\;  Yuan Li,  \;\;   Hui Yang
 }
\maketitle

\begin{center}
\begin{minipage}{140mm}
\begin{center}{\bf Abstract}\end{center}
In this paper, we are devoted to studying the positive solutions of the following higher order Hardy-H\'enon equation
$$
(-\Delta)^{m}u=|x|^{\alpha}u^{p}  \quad\mbox{in}~ B_{1}\setminus\{0\}\subset\mathbb{R}^{n}
$$
with an isolated singularity at the origin, where $\alpha>-2m$, $m\geq1$ is an integer and $n>2m$. For $1<p<\frac{n+2m}{n-2m}$, singularity and decay estimates of solutions will be given. For $\frac{n+\alpha}{n-2m}<p<\frac{n+2m}{n-2m}$ with $-2m<\alpha<2m$, we show the super polyharmonic properties {of solutions near the singularity}, which are essential tools in the study of polyharmonic equation. Using these properties, a classification of isolated singularities of positive solutions is established for the fourth order case, i.e., $m=2$. Moreover, when $m=2$, $\frac{n+\alpha}{n-4}<p<\frac{n+4+\alpha}{n-4}$ and $p\neq \frac{n+4+2\alpha}{n-4}$ with $-4<\alpha\leq0$, we obtain the precise behavior of solutions near the singularity, i.e., either $x=0$ is a removable singularity or
$$\lim_{|x|\rightarrow0}|x|^{\frac{4+\alpha}{p-1}}u(x)=[A_{0}]^{\frac{1}{p-1}},$$
where $A_0>0$ is an exact constant.

\medskip


\medskip


\end{minipage}

\end{center}

\section{Introduction}
In this paper, we are interested in the positive singular solutions of the following higher order elliptic equation
\begin{align}\label{eq1}
(-\Delta)^{m}u(x)=|x|^{\alpha}u^{p}(x)  \quad\mbox{in}~ B_{1} \setminus\{0\}\subset\mathbb{R}^{n},
\end{align}
where $m\geq2$ is an integer, {the punctured unit ball $B_{1} \setminus\{0\}\subset\mathbb{R}^{n}$ with $n>2m$}, $1<p<\frac{n+2m}{n-2m}$ and $\alpha>-2m$. Equation \eqref{eq1} with $m=1$ is traditionally called the H\'enon (resp. Hardy or Lane-Emden) equation if $\alpha>0$ (resp. $\alpha<0$ or $\alpha=0$). In the literature,  equation \eqref{eq1} as a type of classical equations has been widely studied in the last decades, since it can be seen from various geometric and physical problems.

\medskip

To tackle equation \eqref{eq1}, various type of solutions were introduced and studied such as classical solutions, weak solutions, distributional solutions, singular solutions, etc. In this paper, we say that a function $u$ is a solution to \eqref{eq1} if it belongs to $C^{2m}(B_{1}\setminus\{0\})$ and satisfies the equation in the classical sense, and $u$ is a distributional solution to \eqref{eq1} if
$$u\in L^{1}_{loc}(B_{1}\setminus\{0\}), \quad |x|^{\alpha}u^{p}\in L^{1}_{loc}(B_{1}\setminus\{0\}), $$
and \eqref{eq1} is satisfied in the sense of distributions, that is,
$$
\int_{B_{1}}u(-\Delta)^{m}\phi dx=\int_{B_{1}}|x|^{\alpha}u^{p}\phi dx  \qquad   \forall \phi\in C^{\infty}_{c}(B_{1}\setminus\{0\}).
$$
Meanwhile, let us present two important numbers: Sobolev type exponent $P_S(m,\alpha)=\frac{n+2m+2\alpha}{n-2m}~(=\infty,~\text{if}~n=2m)$ and Serrin type exponent $P_C(m,\alpha)=\frac{n+\alpha}{n-2m}~(=\infty,~\text{if}~n=2m)$, which play a vital role in the local behavior of solutions when $p$ passes through them.

\medskip

Now we recall some results of positive singular solutions to the second order equation, i.e., $m=1$. In the autonomous case (i.e., $\alpha=0$), the understanding of asymptotic behavior of positive singular solutions is relatively complete. More precisely, it was studied by Lions \cite{L1980} for $1<p<\frac{n}{n-2}$ and by Gidas-Spruck \cite{GS1981} for $\frac{n}{n-2} < p < \frac{n+2}{n-2}$. When $p=\frac{n+2}{n-2}$, the critical equation $-\Delta u=u^{\frac{n+2}{n-2}}$ is closely related to the classical Yamabe problem, and Caffarelli-Gidas-Spruck \cite{CGS1989} showed that every positive singular solution $u$ satisfies the following asymptotic radial symmetry
$$
u(x)=\bar{u}(|x|)(1+O(|x|)) \quad \mbox{as}~ |x| \rightarrow0,
$$
where $\bar{u}(|x|)=\fint_{\mathbb{S}^{n-1}}u(r, \omega)d\omega$ is the spherical average of $u$. Based on this result, they also established the precise asymptotic behavior of singular solutions. We may also refer to Korevaar-Mazzeo-Pacard-Schoen \cite{KMPS1999} for this critical case.
When $p > \frac{n+2}{n-2}$, the asymptotic behavior was studied by Bidaut-V\'eron and V\'{e}ron in \cite{BVV1991}.

\medskip

In the non-autonomous case, Li \cite{Li} proved the asymptotic radial symmetry of positive solutions to equation \eqref{eq1} with $-2<\alpha\leq 0$ and $1<p\leq \frac{n+2+\alpha}{n-2}$. When $-2<\alpha<2$, for $1<p<\frac{n+\alpha}{n-2}$ Gidas-Spruck \cite[Appendix A]{GS1981} obtained the asymptotic behavior and Zhang-Zhao \cite{ZZ} studied the existence of positive singular solutions; for $p=\frac{n+\alpha}{n-2}$ Aviles \cite{A1987} obtained that either $x=0$ is a removable singularity or
$$\lim_{|x|\rightarrow0}|x|^{n-2}(-\ln|x|)^{\frac{n-2}{\alpha+2}}u(x)=\left(\frac{n-2}{\sqrt{\alpha+2}}\right)^{n-2};$$
for $\frac{n+\alpha}{n-2}<p<\frac{n+2}{n-2}$ and $p\neq\frac{n+2+2\alpha}{n-2}$ Gidas-Spruck \cite{GS1981} showed that the singularity at $x=0$ is removable or there holds
$$\lim_{|x|\rightarrow0}|x|^{\frac{\alpha+2}{p-1}}u(x)=\left[\frac{(2+\alpha)(n-2)}{(p-1)^{2}}\left(p-\frac{n+\alpha}{n-2}\right)\right]^{\frac{1}{p-1}}.$$
 These mean that isolated singularities of positive solutions to the equation \eqref{eq1} {with $m=1$} have been very well understood.

\medskip

However, as far as we know, for the general polyharmonic situation $m\geq 2$, isolated singularities of positive solutions to the equation \eqref{eq1} are less understood. One main difficulty for such higher order equations is the absence of the maximum principle. Recently, this problem is beginning to attract a lot of attention. First we state the results in the autonomous case, i.e., $\alpha=0$. Jin-Xiong \cite{JX2021} proved sharp blow up rates and the asymptotic radial symmetry of positive solutions of \eqref{eq1} with $p=\frac{n+2m}{n-2m}$ under the sign assumptions
\begin{align}\label{SPH}
(-\Delta)^k u>0 \quad \mbox{in} ~ B_1 \setminus \{0\}  \quad\quad \mbox{for all} ~ 1\leq k \leq m-1.
\end{align}
These are also called the {\bf super polyharmonic properties} and are essential tools in the study of polyharmonic equation.  When $m=2$ and $\frac{n}{n-4}<p<\frac{n+4}{n-4}$, Yang \cite{Y2020} showed that either the singularity at $x=0$ is removable or there holds
$$\lim_{|x|\rightarrow0}|x|^{\frac{4}{p-1}}u(x)=[K_{0}(p,n)]^{\frac{1}{p-1}}$$
for an exact constant $K_{0}(p,n)>0$ without the sign assumption of $-\Delta u$. For the supercritical case, the results of singular solutions are less known. Assuming that $|x|^{\frac{4}{p-1}}u(x)\in L^{\infty}(B_{1}(0))$, Wu \cite{W2022} recently studied the classification of isolated singularities for $\frac{n+4}{n-4}<p<\frac{n+4}{n-4}+\varepsilon$, where $\varepsilon>0$ is a constant.

In the non-autonomous case $-2m < \alpha \leq 0$, under the sign assumptions \eqref{SPH}, Caristi-Mitidieri-Soranzo \cite{CMS98} obtained the local behavior of positive solutions for $1< p < \frac{n+\alpha}{n-2m}$ and Li \cite{Lim} proved that every positive solution of \eqref{eq1} is asymptotically radially symmetric near the origin for $\frac{n+\alpha}{n-2m} < p \leq \frac{n+2m+2\alpha}{n-2m}$. Thus, it is natural to ask:

\begin{itemize}
\item [(1)] Can the sign assumptions \eqref{SPH} be removed?

\item [(2)] Can we describe the precise asymptotic behavior of solutions to \eqref{eq1} near the singularity when $\alpha \neq 0$?

\end{itemize}

In this paper, we focus on super polyharmonic properties and asymptotic behavior of positive solutions to the higher order Hardy-H\'enon equation \eqref{eq1}. In particular, when $-2m<\alpha<2m$ and $\frac{n+\alpha}{n-2m}<p<\frac{n+2m}{n-2m}$, we would show that all the positive solutions of \eqref{eq1} with non-removable singularities satisfy the super polyharmonic properties \eqref{SPH} near the origin. Using this as a tool, we establish a Harnack inequality for singular solutions of \eqref{eq1}, and further give a classification of isolated singularities in the fourth order case (i.e., $m=2$). Moreover, for $m=2$ and $-4 < \alpha\leq 0$,  we also obtain the precise asymptotic behavior of positive singular solutions of \eqref{eq1}.

\medskip

We denote $B_r(x)$ as the open ball in $\mathbb{R}^n$ with center $x$ and radius $r$, and also write $B_r(0)$ as $B_r$ for simplicity. First of all, we have the following singularity and decay estimates. Under the sign assumptions \eqref{SPH} and the condition $1< \frac{n+\alpha}{n-2m} < p < \frac{n+2m}{n-2m}$, the singularity estimate in \eqref{Sing048} has also been proved in \cite{Lim}.

\begin{thm}\label{thm1.1}
Let $n>2m$, $\alpha>-2m$ and $1<p<\frac{n+2m}{n-2m}$. Then there exists a constant $C=C(m,n,p,\alpha)>0$ such that the following two conclusions hold.
\begin{itemize}
\item [(i)] Any nonnegative solution of \eqref{eq1} in $B_1 \setminus  \{0\}$ satisfies
\begin{align}\label{Sing048}
\sum_{i\leq 2m-1}|x|^{\frac{2m+\alpha}{p-1}+i}|\nabla^i u(x)|\leq C  \qquad \mbox{for} ~ x \in B_{1/2}\setminus  \{0\}.
\end{align}
\item [(ii)] Any nonnegative solution of \eqref{eq1} in $B^c_1$ satisfies
\begin{align}\label{Dac609}
\sum_{i\leq 2m-1}|x|^{\frac{2m+\alpha}{p-1}+i}|\nabla^i u(x)|\leq C  \qquad \mbox{for} ~ x \in B^c_2.
\end{align}
\end{itemize}
\end{thm}
The origin $x=0$ is called a non-removable singularity of solution $u$ of \eqref{eq1} if $u(x)$ cannot be extended as a continuous function near $0$. Using the above upper bound result, we further show that the super polyharmonic properties hold in some small punctured ball $B_\tau \setminus{\{0}\}$. Remark that the similar super polyharmonic properties of equation \eqref{eq1} on the entire space $\mathbb{R}^n$ (or, on $\mathbb{R}^n \setminus \{0\}$) have been widely studied and applied; see \cite{CL13,CLO06,DQ18,Lin,NY2022,WX1999} and the references therein. However, to the best of our knowledge, such results for the local equation \eqref{eq1} are relatively few.
\begin{thm}\label{thm1.2}
Let $n>2m$, $-2m<\alpha<2m$ and $\frac{n+\alpha}{n-2m}<p<\frac{n+2m}{n-2m}$.  Assume that $u\in C^{2m}(B_{1}\setminus\{0\})$ is a positive solution of \eqref{eq1} and $x=0$ is a non-removable singularity. Then
$$
(-\Delta)^k u > 0 \quad \mbox{for all} ~  1 \leq k \leq m-1
$$
in some small punctured ball $B_\tau \setminus{\{0}\}$ for some $\tau>0.$
\end{thm}

In particular, Theorem \ref{thm1.2} holds for the critical Hardy-Sobolev exponent $p=\frac{n+2m+2\alpha}{n-2m}$ with $-2m < \alpha < 0$.
Moreover, using Theorem \ref{thm1.2} the additional sign conditions \eqref{SPH} in \cite{Lim} can be removed when showing asymptotic radial symmetry of positive solutions. Such super polyharmonic properties allow one to obtain an important Harnack inequality. By employing Harnack inequality and a monotonicity formula, we can establish the following classification of isolated singularities to the fourth order equation.

\begin{thm}\label{thm1.3}
Let $n>4$, $-4<\alpha<4$ and $\frac{n+\alpha}{n-4}<p<\frac{n+4}{n-4}$. Assume that $u\in C^4(B_1(0)\setminus{\{0}\})$ is a positive solution of \eqref{eq1} with $m=2$. Then either the singularity at $x=0$ is removable, or there exist two positive constants $C_1$ and $C_2$ such that
\begin{align*}
C_1 |x|^{-\frac{4+\alpha}{p-1}}\leq u(x)\leq C_2 |x|^{-\frac{4+\alpha}{p-1}} \qquad \mbox{for} ~ x\in B_{1/2}\setminus{\{0}\}.
\end{align*}
\end{thm}

Furthermore, we can show the precise asymptotic behavior of singular solutions to \eqref{eq1} for $m=2$ as follows. Remark that our proof of Theorem \ref{mth1} does not depend on Theorems \ref{thm1.2} and \ref{thm1.3}.
\begin{thm}\label{mth1}
Let $n>4$,  $-4 < \alpha \leq 0$, $\frac{n+\alpha}{n-4}<p<\frac{n+4+\alpha}{n-4}$ and $p\neq \frac{n+4+2\alpha}{n-4}$. Assume that $u\in C^{4}(B_{1}(0)\setminus\{0\})$ is a nonnegative solution of \eqref{eq1} with $m=2$. Then either $x=0$ is a removable singularity or there holds
$$\lim_{|x|\rightarrow0}|x|^{\frac{4+\alpha}{p-1}}u(x)=A_{0}^{\frac{1}{p-1}},$$
where
$$A_{0}:=\frac{4+\alpha}{p-1}\bigg[ \Big(\frac{4+\alpha}{p-1}\Big)^{3}-2(n-4)\Big( \frac{4+\alpha}{p-1} \Big)^{2}+(n^{2}-10n+20)\Big(\frac{4+\alpha}{p-1} \Big)+2(n-2)(n-4)\bigg].
$$
\end{thm}

\medskip

This paper is organized as follows. In Section \ref{sec2}, we recall some classical results which will be used in our arguments. In Section \ref{sec3}, we show the upper bound in Theorem \ref{thm1.1}, and establish the super polyharmonic properties in Theorem \ref{thm1.2}. In Section \ref{sec4}, we prove Theorems \ref{thm1.3} and \ref{mth1} for the fourth order equation.

\section{Preliminaries}\label{sec2}

In this section, for the reader's convenience, we recall some classical results which will be used in our arguments. Let us start with a classification result of polyharmonic equations in Wei-Xu \cite{WX1999}.

\begin{pro}[\cite{WX1999}]\label{pro1}
Suppose that $u$ is a nonnegative classical solution of
$$(-\Delta)^{m}u=u^{p} \quad\mbox{in}~ \mathbb{R}^{n} $$
for $1<p<\frac{n+2m}{n-2m}$. Then $u\equiv0$ in $\mathbb{R}^{n}$.
\end{pro}

\begin{pro}[\cite{Y2021}]\label{pro2}
Let $m\geq1$ be an integer and $-2m<\alpha\leq0$. Suppose that $u\in C(\mathbb{R}^{n}\setminus\{0\})$ is a nonnegative distributional solution of
\begin{align}\label{eq-2}
(-\Delta)^{m}u=|x|^{\alpha}u^{p} \quad \mbox{in} ~ \mathbb{R}^{n}\setminus\{0\}
\end{align}
with $\frac{n+\alpha}{n-2m}\leq p\leq \frac{n+2m+\alpha}{n-2m}$. For $\alpha =0$ and $p=\frac{n+2m}{n-2m}$, suppose also that $u$ has a non-removable singularity at the origin. Then $u$ is radially symmetric with respect to the origin.
\end{pro}

It is worth pointing out that $\frac{n+2m+\alpha}{n-2m}\geq \frac{n+2m+2\alpha}{n-2m}$ holds for $\alpha\leq 0$. For the supercritical case $\frac{n+2m+2\alpha}{n-2m}\leq p\leq \frac{n+2m+\alpha}{n-2m}$, the radial symmetry result in Proposition 2.2 holds via the Kelvin transformation.

We also need the following doubly weighted Hardy-Littlewood-Sobolev inequality, which was showed by Stein-Weiss \cite{SW58}. See also Lieb \cite{L1983}.

\begin{pro}[\cite{L1983,SW58}]\label{pro5}
Let $0<\lambda<n$, $1<p\leq q<\infty$, $0\leq\gamma<\frac{n}{p'}$ (with $\frac{1}{p}+\frac{1}{p'}=1$), $0\leq\beta<\frac{n}{q}$ and $\frac{1}{p}+\frac{\lambda+\gamma+\beta}{n}=1+\frac{1}{q}$. Let the function
$$V(x,y)=|x|^{-\beta}|x-y|^{-\lambda}|y|^{-\gamma}$$
be an integral kernel on $\mathbb{R}^{n}$. Then the map $f\rightarrow Vf:=\int_{\mathbb{R}^n} V(\cdot, y)f(y) dy$ is bounded from $L^{p}(\mathbb{R}^{n})$ to $L^{q}(\mathbb{R}^{n})$.
\end{pro}

Next we recall the doubling lemma from Pol\'a\v{c}ik-Quittner-Souplet \cite{PQS2007}, which is important in the study of singularity and decay estimates.
\begin{pro}[\cite{PQS2007}]\label{pro6}
Suppose that $\emptyset\neq D\subset\Sigma\subset \mathbb{R}^n$, $\Sigma$ is closed and $\Gamma=\Sigma\setminus D$. Let $M: D\rightarrow(0,\infty)$ be bounded on compact subset of $D$. If for a fixed constant $k>0$, there exists $y\in D$ such that
\begin{align}
M(y)dist(y, \Gamma)>2k,\nonumber
\end{align}
then there exists $x\in D$ such that
\begin{align}
M(x)dist(x, \Gamma)>2k, \indent M(x)\geq M(y),\nonumber
\end{align}
and for all $z\in D\cap B_{kM^{-1}(x)}(x)$,
\begin{align}
M(z)\leq 2M(x).\nonumber
\end{align}
\end{pro}

At the end of this section, we state an important regularity lifting result established by Chen-Li \cite{CL2010}, which plays a key role in proving the removable singularity. Let $V$ be a Hausdorff topological vector space. Suppose that there are two extend norms defined in $V$,
$$
\| \cdot \|_{X}, \| \cdot \|_{Y} : V \rightarrow[0,\infty].
$$
Let
$$X:=\{x\in V : ~ \| x \|_{X}<\infty\} \quad \mbox{and} \quad Y:=\{y\in V :  ~ \| y \|_{Y} <\infty \}.$$
Then we have

\begin{pro}[\cite{CL2010}]\label{pro7}
Let $T$ be a contraction map from $X$ into itself and from $Y$ into itself. Assume that $f\in X$ and there exists a function $g\in Z:=X\cap Y$ such that $f=Tf+g$ in $X$. Then there holds $f \in Z$.
\end{pro}

\section{ Upper bound and super polyharmonic properties}\label{sec3}
To prove the upper bound in Theorem \ref{thm1.1}, we need the following estimate which is inspired by \cite{PQS2007,PS2012}.
\begin{lem}\label{lem1}
Let $u$ be a nonnegative solution of
\begin{equation}\label{eq2}
(-\Delta)^{m}u(x)=c(x)u^{p}(x), \quad x\in B_{1}\subset \mathbb{R}^{n},
\end{equation}
where $m\geq1$ is an integer, $n>2m$ and $1<p<\frac{n+2m}{n-2m}$. Assume that $c(x)\in C^{\gamma}(\overline{B_{1}})$ for some $\gamma\in(0,1]$ and satisfies
\begin{equation}\label{decay}
\| c \|_{C^{\gamma}(\overline{B_{1}})}\leq C_{1}, \quad c(x)\geq C_{2}, \quad \forall~x\in B_{1},
\end{equation}
for some constants $C_{1}$, $C_{2}>0$. Then there exists a constant $C=C(\gamma, C_{1}, C_{2}, p, n, m)>0$ such that
$$\sum_{i=0}^{2m-1}|\nabla^{i}u(x)|^{\frac{p-1}{2m+ip-i}}\leq C(1+dist^{-1}(x,\partial B_{1})), \quad \forall~x\in B_{1}.$$
\end{lem}
\begin{proof}
Suppose by contradiction that there exist sequences ${\{u_k}\}$, ${\{c_k}\}$ verifying \eqref{eq2}, \eqref{decay} and a sequence of points ${\{y_k}\}$ such that the functions $M_{k}(y):=\sum\limits_{i=0}^{2m-1}|\nabla^{i}u_{k}(y)|^{\frac{p-1}{2m+ip-i}}$ satisfy
$$M_{k}(y_{k})>2k(1+dist^{-1}(y_{k},\partial B_{1})).$$
Then by the doubling lemma (see Proposition \ref{pro6}), there exists another sequence ${\{x_k}\}$ such that
$$M_{k}(x_{k})\geq M_{k}(y_{k}), \quad M_{k}(x_{k})>2k(1+dist^{-1}(x_{k},\partial B_{1})),$$
and
$$M_{k}(z)\leq2M_{k}(x_{k}) \quad \mbox{for}~  |z-x_{k}|\leq kM^{-1}_{k}(x_{k}).$$
Let $\lambda_{k}:=M^{-1}_{k}(x_{k})$. Then
$$\lambda_{k}\rightarrow0 \quad\mbox{as}~  k\rightarrow\infty $$
due to $M_{k}(x_{k})\geq M_{k}(y_{k})>2k$.

Set
$$v_{k}(y)=\lambda_{k}^{\frac{2m}{p-1}}u(x_{k}+\lambda_{k}y) \quad \mbox{and} \quad \widetilde{c}_{k}(y)=c_{k}(x_{k}+\lambda_{k}y). $$
Then we have
$$\sum\limits_{i=0}^{2m-1}|\nabla^{i}v_{k}(0)|^{\frac{p-1}{2m+ip-i}}=1$$
and
$$ \quad \sum\limits_{i=0}^{2m-1}|\nabla^{i}v_{k}(y)|^{\frac{p-1}{2m+ip-i}}\leq2 \quad\mbox{for}~  |y|\leq k.$$
Moreover,  $v_k$ satisfies
\begin{align*}
(-\Delta)^m v_k(y)=\widetilde{c}_k(y) v^p_k(y), \quad |y|\leq k.
\end{align*}
By \eqref{decay}, we get $C_{2}\leq \widetilde{c}_{k}(y)\leq C_{1}$ and
\begin{equation}\label{c}
|\widetilde{c}_{k}(y)-\widetilde{c}_{k}(z)|\leq C|\lambda_{k}(y-z)|^{\gamma}\leq C|y-z|^{\gamma} \quad\mbox{for $k$ large enough}.
\end{equation}
Using Arzel\`a-Ascoli theorem and combining with \eqref{decay} and \eqref{c}, there exists a constant $c_{0} \geq C_2>0$ such that $\widetilde{c}_{k}\to c_0$ as $k\to +\infty$. It follows from the standard elliptic estimates that, up to a subsequence, $v_{k}\rightarrow v$ in $C_{loc}^{2m}(\mathbb{R}^{n})$ and $v\geq 0$ is a classical solution of
$$(-\Delta)^{m}v=c_{0}v^{p}  \quad\mbox{in}~\mathbb{R}^{n}$$
with $\sum_{i=0}^{2m-1}|\nabla^{i}v(0)|^{\frac{p-1}{2m+ip-i}}=1$. As $1<p<\frac{n+2m}{n-2m}$, this contradicts the Liouville type result in  Proposition \ref{pro1}.
\end{proof}

\medskip

\noindent\textbf{Proof of Theorem \ref{thm1.1}.} ($i$). Assume $0<|x_{0}|<\frac{1}{2}$ and denote $R=\frac{|x_{0}|}{2}$. Then we have that for any $y\in B_{1}$, there holds
$$
\frac{|x_{0}|}{2}<|x_{0}+Ry|<\frac{3|x_{0}|}{2}.
$$
Hence $x_0+Ry\in B_1$. Let us define the function
$$v(y)=R^{\frac{2m+\alpha}{p-1}}u(x_{0}+Ry), \quad  y \in B_1.$$
Then $v$ satisfies
$$(-\Delta)^{m}v(y)=c(y)v^{p}(y),$$
where $c(y)=|y+\frac{x_{0}}{R}|^{\alpha}$. For any $y\in B_{1}$, it is easy to see that
$$1\leq|y+\frac{x_{0}}{R}|\leq3.$$
This implies that for any $y \in B_1$,
$$1\leq c(y)\leq 3^{\alpha} \quad\mbox{and}\quad|\nabla c(y)|\leq C(\alpha),$$
for some constant $C(\alpha)$. It follows from Lemma \ref{lem1} that
$$\sum_{i=0}^{2m-1}|\nabla^{i}v(0)|^{\frac{p-1}{2m+ip-i}}\leq C.$$
Since $x_0 \in B_{1/2} \setminus \{0\}$ is arbitrary, we obtain that
$$\sum_{i=0}^{2m-1}|x|^{\frac{2m+\alpha}{p-1}+i}|\nabla^i u(x)|\leq C, \quad \forall x \in B_{1/2} \setminus \{0\}.$$
This completes the proof of the first part of Theorem \ref{thm1.1}.

($ii$). With the help of Lemma \ref{lem1}, we can follow the above procedure step by step to prove the second part of Theorem \ref{thm1.1}, the detail will be omitted.
\qed

\medskip

Next we will show the super polyharmonic properties in Theorem \ref{thm1.2} based on an integral representation of solutions, which is inspired by the work of Jin-Xiong \cite{JX2021}. Recall that the Green function of $-\Delta$ on the unit ball is given by
$$
G_1(x,y)=\frac{1}{(n-2)\omega_{n-1}}\left(|x-y|^{2-n}-\left|\frac{x}{|x|}-|x|y\right|^{2-n}\right),
$$
where $\omega_{n-1}$ is the surface area of the unit sphere in $\mathbb{R}^n$. Hence, for any $u\in C^2(B_1)\cap C(\overline{B_1})$,
$$
u(x)=\int_{B_1} G_1(x,y)(-\Delta) u(y) dy + \int_{\partial B_1} H_1(x,y) u(y) d\sigma,
$$
where
$$
H_1(x,y)=-\frac{\partial}{\partial \nu_y}G_1(x,y)=\frac{1-|x|^2}{\omega_{n-1} |x-y|^n} \quad \text{for}~ x\in B_1,~y\in \partial B_1.
$$
By induction, we have for $n>2m$ and $u\in C^{2m}(B_1)\cap C^{2m-2}(\overline{B_1})$,
$$
u(x)=\int_{B_1} G_m(x,y)(-\Delta)^m u(y) dy + \sum_{i=1}^m\int_{\partial B_1} H_i(x,y) (-\Delta)^{i-1}u(y) d\sigma,
$$
where
$$
G_m(x,y)=\int_{B_1\times \cdots \times B_1} G_1(x,y_1) G_1(y_1,y_2) \dots G_1(y_{m-1}, y) dy_1 \dots dy_{m-1}
$$
and
$$
H_i(x,y)=\int_{B_1\times \cdots \times B_1} G_1(x,y_1) G_1(y_1,y_2) \dots G_1(y_{i-2}, y_{i-1})H_1(y_{i-1}, y) dy_1 \dots dy_{i-1}.
$$
for $2\leq i\leq m$.

The following result is very similar to \cite[Lemma 2.1]{JX2021}, so the proof is omitted here.
\begin{lem}\label{Lem3.2}
Let $u \in C^{2m}(\overline{B_1}\setminus{\{0}\})$ be a nonnegative solution of
$$
(-\Delta)^m u =|x|^\alpha u^p  \quad \text{in}~B_1\setminus{\{0}\},
$$
where $p>\frac{n+\alpha}{n-2m}$ with $\alpha>-2m$. Then $|x|^\alpha u^p \in L^1(B_1)$ and
$$
u(x)=\int_{B_1} G_m(x,y)|y|^\alpha u^p(y) dy +\sum_{i=1}^m\int_{\partial B_1} H_i(x,y) (-\Delta)^{i-1}u(y) d\sigma \quad \mbox{for} ~ x \in B_1\setminus{\{0}\}.
$$
\end{lem}

\noindent{\bf Proof of Theorem \ref{thm1.2}}. Without loss of generality, we may suppose that $u\in C^{2m}(\overline{B_1 }\setminus \{0\}) $. By Lemma \ref{Lem3.2}, for $k=1,...,m-1$ and $x\in B_1\setminus{\{0}\}$, we have
\begin{align*}
(-\Delta)^k u(x)=&\int_{B_1} G_{m-k}(x,y)|y|^\alpha u^p(y) dy + \sum_{i=k+1}^m \int_{\partial B_1} H_{i-k} (x,y)(-\Delta)^{i-1}u(y) d\sigma\\
=&C\int_{B_1} \frac{|y|^\alpha u^p(y)}{|x-y|^{n-2(m-k)}} dy   + h(x),
\end{align*}
 where $h(x)$ is a smooth function in $B_{1/2}$ and $C=C(n, m)>0$.

{\it Claim}: there holds $$\int_{B_{1/2}}\frac{|y|^\alpha u^p(y)}{|y|^{n+2-2m}} dy =+\infty.$$
 If the claim is true, then there exists $\tau_1\in(0, 1/8)$ small such that
$$
C\int_{B_{1/2}\setminus B_{\tau_1}} \frac{|y|^\alpha u^p(y)}{|y|^{n+2-2m}} dy \geq \| h \|_{L^\infty(B_{1/4})} +1.
$$
Setting $\tau=\frac{\tau_1}{2} >0$, then for $x\in B_\tau \setminus{\{0}\}$ we have
\begin{align*}
(-\Delta)^k u(x)= & C\int_{B_1}\frac{|y|^\alpha u^p(y)}{|x-y|^{n-2(m-k)}} dy + h(x)\\
\geq & C \int_{B_1\setminus B_{\tau_1}}\frac{|y|^\alpha u^p(y)}{|x-y|^{n-2(m-k)}} dy + h(x)\\
\geq & C \int_{B_1\setminus B_{\tau_1}}\frac{|y|^\alpha u^p(y)}{|y|^{n-2(m-k)}} dy + h(x)\\
\geq & C\int_{B_1\setminus B_{\tau_1}}\frac{|y|^\alpha u^p(y)}{|y|^{n-2m+2}} dy + h(x) \geq 1.
\end{align*}
Hence, $(-\Delta)^k u$ is positive in the small punctured ball $B_\tau \setminus \{0\}$.

Next, we show that the claim holds. If not, then this yields
$$
\int_{B_{1/2}} \frac{|y|^\alpha u^p(y)}{|y|^{n+2-2m}} dy < \infty.
$$
By Theorem \ref{thm1.1}, we have $u(x)\leq C |x|^{-\frac{2m+\alpha}{p-1}}$ in $B_{1/2}$. Meanwhile, as $p>\frac{n+\alpha}{n-2m}$ and $-2m<\alpha<2m$, we have $\frac{2m+\alpha}{p-1}<n-2m$. This implies that there exists $q_0 > \frac{n}{n-2m}$ such that $\frac{2m+\alpha}{p -1} q_0 < n$. Therefore, we get
\begin{align*}
\int_{B_{1/2}} u^{q_0} (y) dy \leq C \int_{B_{1/2}} |y|^{-\frac{2m+\alpha}{p-1}q_0} dy<\infty.
\end{align*}
That is, $u\in L^q(B_{1/2})$ for some $q>\frac{n}{n-2m}$.

On the other hand, we want to show $|\cdot|^\alpha u^{p-1}\in L^{\frac{n}{2m}}(B_{1/2})$. First we consider the case $\frac{n+\alpha_+}{n-2m}\leq p<\frac{n+2m}{n-2m}$ with $\alpha_+:= \max{\{0,~\alpha}\}$, which gives $(p-1)\frac{n}{2m}-p \geq 0$. By Theorem \ref{thm1.1}, we obtain
\begin{align*}
\int_{B_{1/2}} \left[|y|^\alpha u^{p-1}(y)\right]^{\frac{n}{2m}} dy
=&\int_{B_{1/2}} |y|^\alpha u^p \left[|y|^{(\frac{n}{2m}-1)\alpha}u^{(p-1)\frac{n}{2m}-p}\right] dy\\
\leq &\int_{B_{1/2}} |y|^\alpha u^p \left[|y|^{(\frac{n}{2m}-1)\alpha} |y|^{-\frac{2m+\alpha}{p-1}[(p-1)\frac{n}{2m}-p]}\right]dy.
\end{align*}
Moreover, note that
\begin{align*}
\frac{2m+\alpha}{p-1}\Big[ (p-1)\frac{n}{2m}-p \Big]-(\frac{n}{2m}-1)\alpha
=&~~\frac{(2m+\alpha)n}{2m} - \frac{p(2m+\alpha)}{p-1}-\frac{\alpha(n-2m)}{2m}\\
=& ~~n-2m - \frac{2m+\alpha}{p-1}\\
<&~~ n+2-2m.
\end{align*}
Hence,
$$
\int_{B_{1/2}} [|y|^\alpha u^{p-1}(y)]^{\frac{n}{2m}} dy \leq \int_{B_{1/2}} \frac{|y|^\alpha u^p(y)}{|y|^{n+2-2m}}dy < \infty.
$$
Thus $|\cdot|^\alpha u^{p-1}\in L^{\frac{n}{2m}}(B_{1/2})$.

Now we consider $\frac{n+\alpha}{n-2m}<p<\frac{n}{n-2m}$ with $-2m<\alpha<0$. In this case, we have $(p-1)\frac{n}{2m}<p$. Denote $s=\frac{p}{(p-1)\frac{n}{2m}}>1$. By H\"older inequality,
\begin{align*}
\int_{B_{1/2}} [|y|^\alpha u^{p-1}(y)]^{\frac{n}{2m}}dy ~=& ~~\int_{B_{1/2}}|y|^{\alpha\frac{n}{2m}+\frac{n+2-2m-\alpha}{s}} \cdot \frac{|y|^{\frac{\alpha}{s}}u^{(p-1)\frac{n}{2m}}(y)}{|y|^{\frac{n+2-2m}{s}}} dy\\
\leq & ~~\left(\int_{B_{1/2}} |y|^{(\alpha\frac{n}{2m}+\frac{n-2m+2-\alpha}{s})s'}dy\right)^{\frac{1}{s'}}\times ~ \left(\int_{B_{1/2}}\frac{|y|^\alpha u^p(y)}{|y|^{n-2m+2}} dy\right)^{\frac{1}{s}},
\end{align*}
where $\frac{1}{s'}+\frac{1}{s}=1$. We only need to show
\begin{align}\label{7}
\left(\alpha\frac{n}{2m}+\frac{n-2m+2-\alpha}{s}\right)s'>-n.
\end{align}
Since
\begin{align*}
\left(\alpha\frac{n}{2m}+\frac{n-2m+2-\alpha}{s}\right)s'
=& ~\frac{n}{2m}\frac{(p-1)(n+2-2m)+\alpha}{p}\cdot \frac{s}{s-1}\\
=& ~\frac{n}{2m} \frac{(p-1)(n+2-2m)+\alpha}{p-(p-1)\frac{n}{2m}},
\end{align*}
the inequality \eqref{7} holds if and only if
\begin{align*}
&\frac{n}{2m} \frac{(p-1)(n+2-2m)+\alpha}{p-(p-1)\frac{n}{2m}}>-n\\
\Longleftrightarrow \quad & ~~(p-1)(n+2-2m)+\alpha > -2m \Big[ p-(p-1)\frac{n}{2m} \Big]\\
\Longleftrightarrow \quad & ~~2p-2+2m +\alpha>0,
\end{align*}
which is true due to $p>1$ and $\alpha>-2m$. Hence, for $\frac{n+\alpha}{n-2m}<p<\frac{n+2m}{n-2m}$ and $-2m<\alpha<2m$, we obtain that $|\cdot|^\alpha u^{p-1} \in L^{\frac{n}{2m}}(B_{1/2})$.

It follows from the regularity result, Corollary 1.1 in Li \cite{LYY}, that $u\in L^\infty(B_{1/2})$ and thus $x=0$ is removable. This contradicts the assumption. The proof of Theorem \ref{thm1.2} is completed.
\qed

\section{Classification and asymptotic behavior}\label{sec4}
In this section, we shall prove Theorems \ref{thm1.3} and \ref{mth1}. The upper bound has been proved in Theorem \ref{thm1.1}. To show the lower bound in Theorem \ref{thm1.3}, an essential tool is a Harnack inequality for singular positive solutions. Such Harnack inequality is a consequence of the super polyharmonic properties in Theorem \ref{thm1.2} and the upper bound in Theorem \ref{thm1.1}. First we recall the following  Harnack inequality for regular solutions of an integral equation in Jin-Li-Xiong \cite{JLX2017}.

\begin{pro}[\cite{JLX2017}]\label{HarR}
For $n\geq1$, $0<\sigma<\frac{n}{2}$, $r>\frac{n}{n-2\sigma}$ and $p>\frac{n}{2\sigma}$, let $0\leq V\in L^{p}(B_{3})$, $0\leq h\in C^{0}(B_{3})$ and $0\leq u\in L^{r}(B_{3})$ satisfy
$$u(x)=\int_{B_{3}}\frac{V(y)u(y)}{|x-y|^{n-2\sigma}}dy+h(x),  \quad x \in B_{2}.$$
If there exists a constant $C_{0}\geq1$ such that $\max_{B_{1}}h\leq C_{0}\min_{B_{1}}h$, then
$$\max_{B_{1}}u\leq C\min_{B_{1}}u, $$
where $C>0$ depends only on $n$, $\sigma$, $C_{0}$, $p$ and an upper bound of $\|V\|_{L^{p}(B_{3})}$.
\end{pro}
\begin{pro}\label{Lem4.1}
Let $n>2m$, $-2m < \alpha < 2m$ and $\frac{n+\alpha}{n-2m}<p<\frac{n+2m}{n-2m}$. Suppose that $u\in C^{2m}(B_1\setminus{\{0}\})$ is a positive solution to \eqref{eq1} with a non-removable singularity at the origin. Then there exist two constants $C>0$ and $r_0 \in (0, 1/16)$ such that for all $0<r<r_0$, we have
$$
\sup_{r/2 \leq |x|\leq 3r/2} u(x) \leq C \inf _{r/2 \leq |x|\leq 3r/2} u(x).
$$
\end{pro}

\begin{proof}
By Lemma \ref{Lem3.2} and the same arguments in \cite[Theorem 1.1]{JX2021} (see the beginning of the proof of Theorem 1.1 there), we can write, modulo a positive constant, that
\begin{align}\label{IE01}
u(x)=\int_{B_{\tau_0}}\frac{|y|^{\alpha}u^{p}(y)}{|x-y|^{n-2m}}dy+h(x)   \quad \mbox{for} ~x \in B_{\tau_0} \setminus \{0\},
\end{align}
where $\tau_0 \in (0, 1/4)$ is small and $h$ is smooth in $B_{\tau_0}$.  Moreover, by using Theorem \ref{thm1.2} we can obtain that $h$ has a positive lower bound in $B_{\tau_0}$. Replacing $u(x)$ by $\big( \frac{\tau_0}{2} \big)^{\frac{2m+\alpha}{p-1}}u \big( \frac{\tau_0}{2} x \big)$,  we may consider the integral equation \eqref{IE01} in $B_2 \setminus \{0\}$ for convenience. Namely,
\begin{align}\label{IE02}
u(x)=\int_{B_{2}}\frac{|y|^{\alpha}u^{p}(y)}{|x-y|^{n-2m}}dy+h(x)   \quad \mbox{for} ~x \in B_{2} \setminus \{0\},
\end{align}
where $u \in C(B_2 \setminus \{0\})$, $|\cdot|^\alpha u^{p} \in L^1(B_2)$, and $h \in C^\infty(B_2)$ is a positive function satisfying
\begin{align}\label{h01}
|\nabla \ln h| \leq C_1 \quad \mbox{in} ~ B_{3/2}.
\end{align}

Let $v(y)=r^{\frac{2m+\alpha}{p-1}}u(ry)$ with $0< r < 1/2$. Then we have, for $y \in B_{2/r} \setminus \{0\}$,
$$
v(y)=\int_{B_{2/r}}\frac{|z|^{\alpha}v^{p}(z)}{|y-z|^{n-2m}}dz+h_{r}(y),
$$
where $h_{r}(y)=r^{\frac{2m+\alpha}{p-1}}h(ry)$. By Theorem \ref{thm1.1} ($i$), we know  that $u(x)\leq C|x|^{-\frac{2m+\alpha}{p-1}}$ and thus $v(y)\leq C$ for $y\in B_{2}\setminus B_{1/10}$. For $|x|=1$, define
$$g_{x}(y)=\int_{B_{2/r}\setminus B_{9/10}(x)}\frac{|z|^{\alpha}v^{p}(z)}{|y-z|^{n-2m}}dz.$$
Then for any $y_{1}, y_{2}\in B_{1/2}(x)$, we have
\begin{align*}
g_{x}(y_{1})&=\int_{B_{2/r}\setminus B_{9/10}(x)}\frac{|y_{2}-z|^{n-2m}}{|y_{1}-z|^{n-2m}}\frac{|z|^{\alpha}v^{p}(z)}{|y_{2}-z|^{n-2m}}dz\\
&\leq C_{n,m}  \int_{B_{2/r}\setminus B_{9/10}(x)}\frac{|z|^{\alpha}v^{p}(z)}{|y_{2}-z|^{n-2m}}dz =C_{n,m}  g_{x}(y_{2}).
\end{align*}
Hence, $g$ satisfies the Harnack inequality in $B_{1/2}(x)$. By \eqref{h01} we have that $h$ also satisfies the Harnack inequality in $B_{1/2}(x)$. Note that
$$v(y)=\int_{B_{9/10}(x)}\frac{|z|^{\alpha}v^{p}(z)}{|y-z|^{n-2m}}dz+g_{x}(y)+h_{r}(y) \quad \mbox{in}~ B_{1/2}(x).$$
It follows from Proposition \ref{HarR} that
$$\sup_{B_{1/2}(x)}v\leq C\inf_{B_{1/2}(x)}v.$$
A standard covering argument leads to
$$\sup_{1/2 \leq |y| \leq 3/2} v(y) \leq  C\inf_{1/2 \leq |y| \leq 3/2} v(y).$$
By rescaling back to $u$, we obtain
$$\sup_{r/2 \leq |x| \leq 3r/2}u(x) \leq C\inf_{r/2 \leq |x| \leq 3r/2 } u(x).$$
This completes the proof.
\end{proof}

Another important tool for proving Theorem \ref{thm1.3} is a monotonicity formula, which will also be used in the proof of Theorem \ref{mth1}.
Let $u$ be a nonnegative solution of equation \eqref{eq1} with $m=2$. Using the Emden-Fowler transformation to the equation \eqref{eq1}, i.e., setting $w(t,\theta)=|x|^{\frac{4+\alpha}{p-1}}u(|x|,\theta)$ with  $t=\ln|x|$ and $\theta=\frac{x}{|x|}$, then the function $w(t,\theta)$ satisfies
\begin{align}\label{eq3}
\partial^{(4)}_tw & +A_{3}\partial^{(3)}_{t}w+A_{2}\partial_{tt}w+A_{1}\partial_{t}w+A_{0}w +\Delta^{2}_{\theta}w \\
& +A_{4}\Delta_{\theta}w +A_{3}\partial_{t}\Delta_{\theta}w+2\partial_{tt}\Delta_{\theta}w=w^{p} \quad \mbox{in} ~ (-\infty, 0) \times \mathbb{S}^{n-1},\nonumber
\end{align}
where $\Delta_\theta$ is the Beltrami-Laplace operator on $\mathbb{S}^{n-1}$ and the coefficients $A_i$ ($i=0, 1, \dots, 4$) are given by
\begin{align}
A_{0}&=B^{4}-2(n-4)B^{3}+(n^{2}-10n+20)B^{2}+2(n-2)(n-4)B,\nonumber\\
A_{1}&=-4B^{3}+6(n-4)B^{2}-2(n^{2}-10n+20)B-2(n-2)(n-4),\nonumber\\
A_{2}&=6B^{2}-6(n-4)B+n^{2}-10n+20,\nonumber\\
A_{3}&=-4B+2n-8,\nonumber\\
A_{4}&=2B^{2}-2(n-4)B - 2(n-4), \nonumber
\end{align}
with
$$
B:=\frac{4+\alpha}{p-1}.
$$
Furthermore, we get that the coefficients $A_{0}$, $A_{1}$ and $A_{3}$ have the following properties.
\begin{lem}\label{SignA}
Let $n\geq5$ and $\alpha > -4$.

\begin{itemize}
\item [(i)] For $\frac{n+\alpha}{n-4}<p<\frac{n+4+2\alpha}{n-4}$, there holds
$$A_{0}>0,\quad A_{1}>0,\quad A_{3}<0.$$

\item [(ii)] For $p > \frac{n+4+2\alpha}{n-4}$,  there holds
$$A_{0}>0,\quad A_{1}<0,\quad A_{3}>0.$$
\end{itemize}
\end{lem}

\begin{Remark}
Note that for $p=\frac{n+4+2\alpha}{n-4}$, equation \eqref{eq1} is critical and conformally invariant. At this moment, we can see that $A_1=A_3=0$.
\end{Remark}

\begin{proof}
 Direct calculations give that
\begin{equation*}
B\in
\begin{cases}
\left(\frac{n-4}{2}, ~n-4\right) \quad & \text{for}~\frac{n+\alpha}{n-4}<p<\frac{n+4+2\alpha}{n-4}, \\
\left(0, ~ \frac{n-4}{2}\right) \quad & \text{for}~\frac{n+4+2\alpha}{n-4}<p.
\end{cases}
\end{equation*}
Let
$$f(s)=-4s+2n-8.$$
Then for $s>\frac{n-4}{2}~(resp. <)$, we have $f(s)<0~(resp. >0)$. This implies that $A_{3}<0~(resp.>0)$ for $\frac{n+\alpha}{n-4}<p<\frac{n+4+2\alpha}{n-4}~(resp.  \frac{n+4+2\alpha}{n-4}<p )$.

To judge the sign of $A_1$, consider the function
$$g(s)=-4s^{3}+6(n-4)s^{2}-2(n^{2}-10n+20)s-2(n-2)(n-4). $$
Then
$$g'(s)=-12s^{2}+12(n-4)s-2(n^{2}-10n+20).$$
Two roots of $g'(s)=0$ are given by $s_1=\frac{3(n-4) + \sqrt{3(n-2)^2+12}}{6}$ and $s_2=\frac{3(n-4) - \sqrt{3(n-2)^2+12}}{6}$. When $n\geq 8$, we have $n^2-10n +20>0$ and this yields
$$
\frac{n-4}{2}<s_1<n-4,\quad \quad 0<s_2<\frac{n-4}{2}.
$$
When $5\leq n\leq 7$, we have $n^2-10n+20<0$ and hence
$$
s_1>n-4,\quad\quad s_2 < 0.
$$
These imply that
$$
g(s) > \min{\{ g(\frac{n-4}{2}), ~ g(n-4)} \}  \quad \forall~s\in (\frac{n-4}{2}, n-4),
$$
and
$$ g(s) <  \max\{ g(0), ~ g(\frac{n-4}{2}) \} \quad \forall~ s \in (0, \frac{n-4}{2}).$$
On the other hand, direct calculations give
$$g(0)<0, \quad g(\frac{n-4}{2})=0 \quad \mbox{and}\quad g(n-4)>0.$$
Hence, we obtain that $A_{1}>0$ for $\frac{n-4}{2}<B<n-4$ and $A_1<0$ for $0<B<\frac{n-4}{2}$.

Finally, let's determine the sign of $A_0$. Since
\begin{align}
A_{0}&=B^{4}-2(n-4)B^{3}+(n^{2}-10n+20)B^{2}+2(n-2)(n-4)B \nonumber\\
&=B[B^{3}-2(n-4)B^{2}+(n^{2}-10n+20)B+2(n-2)(n-4)], \nonumber
\end{align}
we consider the function
$$
h(s)=s^{3}-2(n-4)s^{2}+(n^{2}-10n+20)s+2(n-2)(n-4)
$$
for $0<s<n-4$.  Notice that
$$h'(s)=3s^{2}-4(n-4)s+(n^{2}-10n+20)$$
and $h'(n-4)=4-2n<0$. Then we have
$$
h(s) > \min\{ h(0), h(n-4) \}  \quad \forall~ s \in (0, n-4).
$$
It is easy to check that
$$
h(0)>0 \quad \text{and} \quad h(n-4)=0.
$$
This implies $A_{0}>0$ for all $p > \frac{n+\alpha}{n-4}$.
\end{proof}

Next, we introduce an important energy function $E(t,w)$ defined by
\begin{align}\label{Monf}
E(t;w):= & \int_{\mathbb{S}^{n-1}}\partial^{(3)}_{t}w\partial_{t}w d\sigma-\frac{1}{2}\int_{\mathbb{S}^{n-1}}\left[(\partial_{t}w)^{2}-2A_{3}\partial_{t}^{2}w\partial_{t}w-A_{2}(\partial_{t}w)^{2}\right]d\sigma  \nonumber \\
& +\frac{1}{2}\int_{\mathbb{S}^{n-1}}\left[|\Delta_{\theta}w|^{2}-A_{4}|\nabla_{\theta}w|^{2}\right]d\sigma +\frac{A_{0}}{2}\int_{\mathbb{S}^{n-1}}w^{2}d\sigma -\frac{1}{p+1}\int_{\mathbb{S}^{n-1}}w^{p+1}d\sigma  \nonumber \\
& -\int_{\mathbb{S}^{n-1}}|\partial_{t}\nabla_{\theta}w|^{2}d\sigma.
\end{align}
Combining the equation \eqref{eq3} and the singularity estimate in Theorem \ref{thm1.1}, we obtain the following monotonicity formula for $E(t;w)$ and the uniform boundedness for $w$. The proofs are essentially the same as those in \cite{Y2020}, so we omit them here.
\begin{pro}\label{Mon01}
Suppose that $n\geq5$, $\alpha > -4$ and $w$ is a nonnegative $C^{4}$ solution of \eqref{eq3}. Then $E(t;w)$ is non-increasing (resp. non-decreasing) in $t\in(-\infty,0)$ for $\frac{n+\alpha}{n-4} < p \leq \frac{n+4+2\alpha}{n-4}$ (resp. $p > \frac{n+4+2\alpha}{n-4} $). Furthermore, we have
\begin{align}\label{10}
\frac{d}{dt}E(t;w)=A_{3}\int_{\mathbb{S}^{n-1}}\left[(\partial_{tt}w)^{2}+|\partial_{t}\nabla_{\theta}w|^{2}\right]d\sigma-A_{1}\int_{\mathbb{S}^{n-1}}(\partial_{t}w)^{2} d\sigma.
\end{align}
\end{pro}

\begin{pro}\label{Mon02}
Let $w$ be a nonnegative $C^{4}$ solution of \eqref{eq3} with $1<p<\frac{n+4}{n-4}$. Then $w$, $\partial_{t}w$, $\partial_{tt}w$, $\partial_{ttt}w$, $\Delta_{\theta}w$ and $|\nabla_{\theta}w|$ are uniformly bounded in $(-\infty, -\ln2)\times\mathbb{S}^{n-1}$.
\end{pro}

From the above two propositions, we know that $\lim_{t\rightarrow-\infty}E(t;w)$ exists for $\frac{n+\alpha}{n-4}<p<\frac{n+4}{n-4}$ and $-4< \alpha <4$. Define
\begin{align}\label{1u067}
\widetilde{E}(r;u):=E(t;w)
\end{align}
with $t=\ln r$. Then we get
$$\widetilde{E}(0;u) := \lim_{r \rightarrow 0^+}\widetilde{E}(r;u) = \lim_{t\rightarrow-\infty}E(t;w).$$
For any $\lambda>0$, define $u^{\lambda}(x)=\lambda^{\frac{4+\alpha}{p-1}}u(\lambda x)$. Then the function $u^\lambda(x)$ satisfies
$$\Delta^{2}u^{\lambda}(x)=|x|^{\alpha}(u^{\lambda}(x))^{p}  \quad\mbox{in}~ B_{1/\lambda}(0)\setminus\{0\},$$
and
$$\widetilde{E}(r;u^{\lambda})=\widetilde{E}(r\lambda;u).$$

Now, using the Harnack inequality in Proposition \ref{Lem4.1} and the monotonicity formula in Proposition \ref{Mon01}, we show the following result which is important in establishing the lower bound of singular solutions.

\begin{pro}\label{lim}
Let $m=2$,  $-4<\alpha<4$ and $\frac{n+\alpha}{n-4}<p<\frac{n+4}{n-4}$. Suppose that $u\in C^{4}(B_1\setminus{\{0}\})$ is a positive solution to \eqref{eq1} with a non-removable singularity at the origin. If
$$\liminf_{|x|\rightarrow0}|x|^{\frac{4+\alpha}{p-1}}u(x)=0,$$
then
$$\lim_{|x|\rightarrow0}|x|^{\frac{4+\alpha}{p-1}}u(x)=0.$$
\end{pro}

\begin{proof} We show the proof by discussing two cases separately: $p\neq\frac{n+4+2\alpha}{n-4}$ and $p = \frac{n+4+2\alpha}{n-4}$.

\textbf{Case 1:} For $\frac{n+\alpha}{n-4}<p<\frac{n+4}{n-4}$ and $p\neq\frac{n+4+2\alpha}{n-4}$. In this case, we need to use some arguments in Yang-Zou \cite{YZ} and divide the proof into three steps.

{\it Step 1}. {\it If $\liminf\limits_{|x|\rightarrow0}|x|^{\frac{4+\alpha}{p-1}}u(x)=0$, then there holds $\lim\limits_{r\rightarrow0^{+}}\widetilde{E}(r;u)=0$. }
Since $\liminf\limits_{|x|\rightarrow0}|x|^{\frac{4+\alpha}{p-1}}u(x)=0$, there exists a sequence of points ${\{x_i}\}$ such that $x_{i}\rightarrow0$ and
$$
|x_{i}|^{\frac{4+\alpha}{p-1}}u(x_{i})\rightarrow0  \quad \mbox{as}~ i\rightarrow\infty.
$$
Let $r_{i}:=|x_{i}|$. Using the Harnack inequality in Proposition \ref{Lem4.1}, we have
\begin{align}\label{asy01}
r_{i}^{\frac{4+\alpha}{p-1}}u(r_{i}e_{1})\rightarrow0  \quad \mbox{as}~ i\rightarrow\infty,
\end{align}
where $e_{1}=(1,0,\cdots,0)\in\mathbb{R}^{n}$. Define $v_{i}(x)=r_{i}^{\frac{4+\alpha}{p-1}}u(r_{i}x)$, then the function $v_i(x)$ satisfies
$$\Delta^{2}v_{i}(x)=|x|^{\alpha}v_{i}^{p}(x)  \quad  x\in B_{ 2/r_{i} }\setminus\{0\}.$$
It follows from Theorem \ref{thm1.1} and Proposition \ref{Lem4.1} that $v_{i}$ is locally uniformly bounded away from the origin. By elliptic estimates and Theorem \ref{thm1.2} there exists $v \in C^4(\mathbb{R}^n\setminus \{0\})$ such that $v_{i}\rightarrow v$ in $C_{loc}^4(\mathbb{R}^n\setminus \{0\})$ as $i\to \infty$,  and $v$ satisfies
$$\begin{cases}
\Delta^{2}v=|x|^{\alpha}v^{p},  \\
-\Delta v\geq 0
\end{cases} \quad \mbox{in} ~ \mathbb{R}^{n}\setminus\{0\}.
$$
By \eqref{asy01}, we have $v(e_{1})=0$. The maximum principle implies that $v\equiv0$ in $\mathbb{R}^{n} \setminus \{0\}$. On the other hand, since $\widetilde{E}(r;u)$ is invariant under the scaling, we get
$$\lim_{i\rightarrow\infty}\widetilde{E}(r_{i};u)=\lim_{i\rightarrow\infty}\widetilde{E}(1;v_{i})=\widetilde{E}(1;v)=0.$$
By the monotonicity in Proposition \ref{Mon01}, we obtain
$$\lim_{r\rightarrow0^{+}}\widetilde{E}(r;u)=0.$$

{\it Step 2}. {\it Let $v \in C^4(\mathbb{R}^n \setminus \{0\})$ be a nonnegative solution of $\Delta^{2}v=|x|^{\alpha}v^{p}$ in $\mathbb{R}^{n}\setminus\{0\}$. If $\widetilde{E}(r;v)=0$ for $r\in (0, +\infty)$, then $v\equiv0$ in $\mathbb{R}^{n}\setminus \{0\}$.}

By the definition of $\widetilde{E}(r;v)$ in \eqref{1u067}, we have $E(t;w)=0$ for $t\in (-\infty,\infty)$, where $w(t,\theta)=|x|^{\frac{4+\alpha}{p-1}}v(|x|,\theta)$ with $t=\ln |x|$. This implies that $\frac{d}{dt}E(t;w)=0$.  It follows from Proposition \ref{Mon01} and Lemma \ref{SignA} that $\partial_{t}w=0$ due to $p \neq \frac{n+4+2\alpha}{n-4}$. Therefore \eqref{eq3} can be reduced to
$$A_{0}w+\Delta^{2}_{\theta}w+A_{4}\Delta_{\theta}w=w^{p}.$$
Integrating on $\mathbb{S}^{n-1}$,  we get
$$\int_{\mathbb{S}^{n-1}}A_{0}w^{2}+|\Delta_{\theta}w|^{2}-A_{4}|\nabla_{\theta}w|^{2}dS=\int_{\mathbb{S}^{n-1}}w^{p+1}dS.$$
On the other hand,  $E(t;w)=0$ implies
$$\frac{1}{2}\int_{\mathbb{S}^{n-1}}|\Delta_{\theta}w|^{2}-A_{4}|\nabla_{\theta}w|^{2}dS+\frac{A_{0}}{2}\int_{\mathbb{S}^{n-1}}w^{2}dS-\frac{1}{p+1}\int_{\mathbb{S}^{n-1}}w^{p+1}dS=0.$$
Therefore, we obtain
$$(1-\frac{2}{p+1})\int_{\mathbb{S}^{n-1}}w^{p+1}dS=0.$$
This means that $w\equiv0$ on $\mathbb{S}^{n-1}$ and hence $v\equiv0$ in $\mathbb{R}^n \setminus \{0\}$.

{\it Step 3}. Define $u^{\lambda}(x)=\lambda^{\frac{4+\alpha}{p-1}}u(\lambda x)$ with $\lambda>0$ small. It follows from Theorem \ref{thm1.1} and Proposition \ref{Lem4.1} that $u^{\lambda}$ is locally uniformly bounded away from the origin. Hence, there is a subsequence ${\{\lambda_{i}}\}$ such that $\lambda_i\rightarrow0$ and $u^{\lambda_{i}}\rightarrow u^{0}$ in $C_{loc}^4(\mathbb{R}^n \setminus \{0\})$ as $i\to \infty$.   Moreover, $u^0$ is nonnegative and satisfies
$$\Delta^{2}u^{0}=|x|^{\alpha}(u^{0})^{p},  \quad x\in\mathbb{R}^{n}\setminus\{0\}.$$
The scaling invariance of $\widetilde{E}(r;u)$ leads to
$$\widetilde{E}(r;u^{0})=\lim_{i\rightarrow\infty}\widetilde{E}(r;u^{\lambda_{i}})=\lim_{i\rightarrow\infty}\widetilde{E}(r\lambda_{i};u)=\lim_{r\rightarrow0}\widetilde{E}(r;u)=0 \quad \mbox{for all} ~ 0 < r <\infty.
$$
By the conclusion in Step 2, we obtain $u^{0} \equiv 0$. In particular,
$$\lim_{\lambda\rightarrow0}\lambda^{\frac{4+\alpha}{p-1}}u(\lambda x)=0$$
 uniformly for $x\in \partial B_{1}$, which immediately implies
 $$\lim_{|x|\rightarrow0}|x|^{\frac{4+\alpha}{p-1}}u(x)=0.$$

\textbf{Case 2:} For $p=\frac{n+4+2\alpha}{n-4}$. Assume by contradiction that $\limsup_{|x|\rightarrow0}|x|^{\frac{4+\alpha}{p-1}}u(x)=C>0$. Then there exists a sequence of points $\{x_i\}$ such that $s_i:=|x_i| \to 0$ and $u(x_i) \to \infty$ as $i \to \infty$. By the Harnack inequality in Proposition \ref{Lem4.1}, we have
$$
\inf_{\partial B_{s_i}} u \geq C^{-1} u(x_i)  \to \infty.
$$
It follows from Theorem \ref{thm1.2} that $-\Delta u \geq 0$ in $B_\tau \setminus \{0\}$ for some $\tau>0$. Using the maximum principle, we get
$$
\inf_{B_{s_i} \setminus B_{s_{i+1}}} u  \geq \min \left\{ \inf_{\partial B_{s_i}} u,  \inf_{\partial B_{s_{i+1}}} u \right\} \to \infty,
$$
and hence
$$\liminf_{|x|\rightarrow0}u(x)=\infty.$$
Moreover, from the assumptions and Proposition \ref{Lem4.1}, we can find $r_{i}\rightarrow0$ such that $r_{i}^{\frac{4+\alpha}{p-1}}\overline{u}(r_{i})\rightarrow0$ as $i\to \infty$,  and $r_{i}$ is a local minimum point of $r^{\frac{4+\alpha}{p-1}}\overline{u}(r)$ with $\overline{u}(r)=\frac{1}{|\partial B_{r}|}\int_{\partial B_{r}}udS$. As in the proof of Proposition \ref{Lem4.1}, we may suppose that $u$ satisfies the integral equation
$$
u(x)=\int_{B_{2}}\frac{|y|^{\alpha}u^{p}(y)}{|x-y|^{n-4}}dy+h(x)   \quad \mbox{for} ~x \in B_{2} \setminus \{0\},
$$
where $u \in C^4(B_2 \setminus \{0\})$, $|\cdot|^\alpha u^{p} \in L^1(B_2)$ and $h \in C^\infty(B_2)$ is a positive function. Now we use some arguments from \cite[Proposition 4.5]{JX2021}. Define
$$\varphi_{i}(y)=\frac{u(r_{i}y)}{u(r_{i}e_{1})},$$
where $e_{1}=(1,0,\cdots,0)\in\mathbb{R}^{n}$. Then $\varphi_{i}$ satisfies
\begin{align}\label{vint0}
\varphi_{i}(y)=\int_{B_{2/r_{i} }}\frac{\left(r_{i}^{\frac{4+\alpha}{p-1}}u(r_{i}e_{1})\right)^{p-1}|z|^{\alpha}\varphi_{i}^{p}(z)}{|y-z|^{n-4}}dz+h_{i}(y), \quad y \in B_{2/r_{i}} \setminus \{0\},
\end{align}
where $h_{i}(y)=u(r_{i}e_{1})^{-1}h(r_{i}y)\rightarrow0$ in $C^{4}_{loc}(\mathbb{R}^{n})$ thanks to  $u(r_{i}e_{1})\rightarrow\infty$  as $i \to \infty$.

It follows from the Harnack inequality in Proposition \ref{Lem4.1} that $r_{i}^{\frac{4+\alpha}{p-1}}u(r_{i}e_{1})\rightarrow0$ and $\varphi_{i}$ is locally uniformly bounded in $B_{2/{r_{i}}}\setminus\{0\}$. Hence
\begin{align}\label{loc0}
\left(r_{i}^{\frac{4+\alpha}{p-1}}u(r_{i}e_{1})\right)^{p-1}|z|^{\alpha}\varphi_{i}^{p}(z)\rightarrow0 \quad \mbox{in} ~ C_{loc}(\mathbb{R}^n \setminus \{0\}).
\end{align}
For any $t>1$, $0<|y|<t$ and $0<\varepsilon<\frac{|y|}{100}$, we have, after passing to a subsequence,
\begin{align*}
& \lim_{i\rightarrow\infty}\int_{B_{t}}\frac{\left(r_{i}^{\frac{4+\alpha}{p-1}}u(r_{i}e_{1})\right)^{p-1}|z|^{\alpha}\varphi_{i}^{p}(z)}{|y-z|^{n-4}}dz \\
&~~ =\lim_{i\rightarrow\infty}\int_{B_{\varepsilon}}\frac{\left(r_{i}^{\frac{4+\alpha}{p-1}}u(r_{i}e_{1})\right)^{p-1}|z|^{\alpha}\varphi_{i}^{p}(z)}{|y-z|^{n-4}}dz\\
&~~ =|y|^{4-n}(1+O(\varepsilon))\lim_{i\rightarrow\infty}\int_{B_{\varepsilon}}\left(r_{i}^{\frac{4+\alpha}{p-1}}u(r_{i}e_{1})\right)^{p-1}|z|^{\alpha}\varphi_{i}^{p}(z)dz.
\end{align*}
Sending $\varepsilon\rightarrow0$, we obtain
$$\lim_{i\rightarrow\infty}\int_{B_{t}}\frac{\left(r_{i}^{\frac{4+\alpha}{p-1}}u(r_{i}e_{1})\right)^{p-1}|z|^{\alpha}\varphi_{i}^{p}(z)}{|y-z|^{n-4}}dz=\frac{a}{|y|^{n-4}} $$
for some constant $a\geq0$. On the other hand, by the equation \eqref{vint0} we have
$$\lim_{i\rightarrow\infty}\int_{B_{2/{r_{i}}}\setminus B_{t}}\frac{\left(r_{i}^{\frac{4+\alpha}{p-1}}u(r_{i}e_{1})\right)^{p-1}|z|^{\alpha}\varphi_{i}^{p}(z)}{|y-z|^{n-4}}dz\rightarrow f(y) \geq 0 \quad \mbox{in} ~ C_{loc}^4(B_t)$$
for some  function $f\in C^{4}(B_{t})$.  We claim that $f$ is a constant function in $B_t$. Indeed, for any fixed large $R>0$ and $y\in B_{t}$, it follows from \eqref{loc0} that
$$\lim_{i\rightarrow\infty}\int_{t\leq|z|\leq R}\frac{\left(r_{i}^{\frac{4+\alpha}{p-1}}u(r_{i}e_{1})\right)^{p-1}|z|^{\alpha}\varphi_{i}^{p}(z)}{|y-z|^{n-4}}dz\rightarrow0  \quad \mbox{as}~ i\rightarrow\infty.$$
For any $y',y''\in B_{t}$, we have
\begin{align*}
\int_{B_{2/{r_{i}}}\setminus B_{R}} & \frac{\left(r_{i}^{\frac{4+\alpha}{p-1}}u(r_{i}e_{1})\right)^{p-1}|z|^{\alpha}\varphi_{i}^{p}(z)}{|y'-z|^{n-4}}dz \\
& \leq \left(\frac{R+t}{R-t}\right)^{n-4}\int_{B_{2/{r_{i}}}\setminus B_{R}}\frac{\left(r_{i}^{\frac{4+\alpha}{p-1}}u(r_{i}e_{1})\right)^{p-1}|z|^{\alpha}\varphi_{i}^{p}(z)}{|y''-z|^{n-4}}dz.
\end{align*}
Therefore,
$$f(y')\leq\left(\frac{R+t}{R-t}\right)^{n-4}f(y'').$$
By sending $R\rightarrow\infty$, we get
$$f(y) = f(0) \quad \mbox{for all}~ y\in B_{t}.$$
Note that $\varphi_i$ is locally uniformly bounded in $C^{5}(B_{2/{r_i}}\setminus{\{0}\})$. Thus, up to a subsequence, we have
$$\varphi_{i}(y)\rightarrow \varphi(y):=\frac{a}{|y|^{n-4}}+f(0) \quad \text{in}~ C_{loc}^4(\mathbb{R}^n\setminus{\{0}\})~ \text{as} ~ i\to  \infty.$$
Since $\varphi_{i}(e_{1})=1$ and $\frac{d}{dr}(r^{\frac{4+\alpha}{p-1}}\overline{\varphi}_{i}(r)) \big|_{r=1}=0$, we have $\varphi(e_{1})=1$ and
$$
\frac{d}{dr}(r^{\frac{4+\alpha}{p-1}}\overline{\varphi}(r)) \Big|_{r=1}=0.
$$
These imply that
$$a=\frac{4+\alpha}{(p-1)(n-4)}=\frac{1}{2} \quad\mbox{and}\quad f(0)=1- a =\frac{1}{2}.$$

Since $|\nabla^{k}\varphi_{i}|\leq C$ near $\partial B_{1}$ and $r_{i}^{\frac{4+\alpha}{p-1}}u(r_{i}e_{1})=o(1)$, we obtain
$$
|\nabla^{k}u(x)|\leq Cr_{i}^{-k}u(r_{i}e_{1})=o(1)r_{i}^{-\frac{4+\alpha}{p-1}-k} \quad\mbox{for all }|x|=r_{i}, ~ k=0,1,2,3.
$$
Set $w(t,\theta)=|x|^{\frac{4+\alpha}{p-1}}u(|x|,\theta)$ with  $t=\ln|x|$ and $\theta=\frac{x}{|x|}$. Then the function $w$ satisfies \eqref{eq3} and
$$
w, w_{t}, w_{tt}, w_{ttt}, |\nabla_{\theta}w|, \Delta_{\theta}w=o(1) \quad \mbox{for} ~ t_i= \ln r_i.
$$
Let $E(t;w)$ be defined as in \eqref{Monf}. Then
$$\lim_{i\rightarrow \infty}E(t_i;w)=0.$$
Using Proposition \ref{Mon01}, we see that $\frac{d}{dt}E(t;w)=0$ due to $A_1=A_3=0$ in the case of $p=\frac{n+4+2\alpha}{n-4}$. Hence
$$E(t;w)=0 \quad\mbox{for all}~  t\in(-\infty, \infty).$$
Set $\widetilde{\varphi}_{i}(t,\theta):=|x|^{\frac{4+\alpha}{p-1}}\varphi_{i}(|x|,\theta)$ with $t=\ln |x|$. Then we have
$$\widetilde{\varphi}_{i}(t,\theta)\rightarrow\widetilde{\varphi}(t,\theta):=\frac{1}{2} \left[ e^{(\frac{4+\alpha}{p-1}+4-n)t}+  e^{\frac{4+\alpha}{p-1}t} \right]$$
and
$$r_{i}^{\frac{4+\alpha}{p-1}}u(r_{i}e_{1})\widetilde{\varphi}_{i}(t,\theta)=w(t+\ln r_{i},\theta).$$
Therefore,
$$0=E(t;w)=E(t+\ln r_{i};w)=E(t;w(t+\ln r_{i},\theta))=E(t;r_{i}^{\frac{4+\alpha}{p-1}}u(r_{i}e_{1})\widetilde{\varphi}_{i}(t,\theta)),$$
that is,
\begin{align*}
0= & \int_{\mathbb{S}^{n-1}}(\widetilde{\varphi}_{i}(t,\theta))_{ttt}(\widetilde{\varphi}_{i}(t,\theta))_{t}dS-\frac{1}{2}\int_{\mathbb{S}^{n-1}}(1-A_{2})[(\widetilde{\varphi}_{i}(t,\theta))_{t}]^{2}dS +\frac{A_{0}}{2}\int_{\mathbb{S}^{n-1}}\widetilde{\varphi}_{i}^{2}(t,\theta)dS\\
& + \int_{\mathbb{S}^{n-1}}A_{3}(\widetilde{\varphi}_{i}(t,\theta))_{tt}(\widetilde{\varphi}_{i}(t,\theta))_{t}dS
+\frac{1}{2}\int_{\mathbb{S}^{n-1}}|\Delta_{\theta}\widetilde{\varphi}_{i}(t,\theta)|^{2}-A_{4}|\nabla_{\theta}\widetilde{\varphi}_{i}(t,\theta)|^{2}dS\\
& -\frac{1}{p+1}\int_{\mathbb{S}^{n-1}}\left(r_{i}^{\frac{4+\alpha}{p-1}}u(r_{i}e_{1})\right)^{p-1}(\widetilde{\varphi}_{i}(t,\theta))^{p+1}dS-\int_{\mathbb{S}^{n-1}}|\partial_{t}\nabla_{\theta}\widetilde{\varphi}_{i}(t,\theta)|^{2}dS.
\end{align*}
Letting $i\rightarrow\infty$, we get
\begin{align*}
0=&\int_{\mathbb{S}^{n-1}}\widetilde{\varphi}_{ttt}\widetilde{\varphi}_{t}dS-\frac{1}{2}\int_{\mathbb{S}^{n-1}}(1-A_{2})[\widetilde{\varphi}_{t}]^{2}dS+\int_{\mathbb{S}^{n-1}}A_{3}\widetilde{\varphi}_{tt}\widetilde{\varphi}_{t}dS+\frac{A_{0}}{2}\int_{\mathbb{S}^{n-1}}\widetilde{\varphi}^{2}dS.
\end{align*}
Taking $t=0$, we have
$$\widetilde{\varphi}_{t}|_{t=0}=\frac{1}{2} \left[ \Big( \frac{4+\alpha}{p-1}+4-n \Big)+ \frac{4+\alpha}{p-1} \right]=0$$
and
$$\widetilde{\varphi}_{ttt}|_{t=0}=\frac{1}{2} \left[ \Big( \frac{4+\alpha}{p-1}+4-n \Big)^{3}+ \Big( \frac{4+\alpha}{p-1} \Big)^{3} \right]=0.$$
Hence, we get
\begin{align*}
0 & = \left[ \int_{\mathbb{S}^{n-1}}\widetilde{\varphi}_{ttt}\widetilde{\varphi}_{t}dS - \frac{1}{2}\int_{\mathbb{S}^{n-1}}(1-A_{2})(\widetilde{\varphi}_{t})^{2}dS+\int_{\mathbb{S}^{n-1}}A_{3}\widetilde{\varphi}_{tt}\widetilde{\varphi}_{t}dS+\frac{A_{0}}{2}\int_{\mathbb{S}^{n-1}}\widetilde{\varphi}^{2}dS]\right] \Bigg|_{t=0}\\
&=\frac{A_{0}}{2}|\mathbb{S}^{n-1}|>0.
\end{align*}
This is a contradiction. The proof of Proposition \ref{lim} is completed.
\end{proof}

\medskip

To complete the proof of Theorem \ref{thm1.3}, we also need the following sufficient condition for removability of isolated singularities.  This will also be used in the proof of Theorem \ref{mth1}.

\begin{pro}\label{Remov}
Let $m=2$,  $-4<\alpha<4$ and $\frac{n+\alpha}{n-4}<p<\frac{n+4}{n-4}$. Suppose that $u\in C^{4}(B_1\setminus{\{0}\})$ is a positive solution to \eqref{eq1}.  If
\begin{align}\label{step1.1}
\lim_{|x|\rightarrow0}|x|^{\frac{4+\alpha}{p-1}}u(x)=0,
\end{align}
then the singularity at $x=0$ is removable, i.e., $u$ can be extended as a continuous function near the origin $0$.
\end{pro}

\begin{proof}
First, we claim that there holds
\begin{align}\label{Remty01}
\int_{B_{1/2}}u(x) |x|^{-n + \frac{4+\alpha}{p-1}} dx < \infty.
\end{align}
Set $\psi(x)=|x|^{-\gamma}$ with $\gamma=n-4-\frac{4+\alpha}{p-1} >0$. Then
$$\Delta^{2}\psi(r)=\gamma(\gamma+2)(\gamma-n+2)(\gamma-n+4)r^{-\gamma-4}.$$
Denote
$$\Lambda:=\gamma(\gamma+2)(\gamma-n+2)(\gamma-n+4).$$
Then we have $\Lambda>0$ and
\begin{align}
\frac{\Delta^{2}\psi}{\psi}=\Lambda|x|^{-4} \quad\mbox{in}~ \mathbb{R}^{n}\setminus\{0\}.\nonumber
\end{align}
For any small $\varepsilon>0$, we define a smooth cut-off function as follows:
\begin{eqnarray*}
\zeta_{\varepsilon}(x)=
\begin{cases}
1, &\varepsilon\leq|x|\leq\frac{1}{2},\cr
0, &|x|\leq\frac{\varepsilon}{2} ~ \text{or} ~ |x|\geq\frac{3}{4},
\end{cases}
\end{eqnarray*}
\begin{align}\label{step2.2}
|\nabla^{k}\zeta_{\varepsilon}(x)|\leq C\varepsilon^{-k} \quad\mbox{for}~ \frac{\varepsilon}{2}\leq|x|\leq\varepsilon,
\end{align}
and
\begin{align}\label{step2.3}
|\nabla^{k}\zeta_{\varepsilon}(x)|\leq C\left(\frac{1}{4}\right)^{-k} \quad\mbox{for}~ \frac{1}{2}\leq|x|\leq\frac{3}{4},
\end{align}
where $C>0$ is independent of $\varepsilon$. Taking $\zeta_{\varepsilon}(x)\psi(x)$ as a test function in \eqref{eq1}, we have
\begin{align}
\int_{B_{1}}u\zeta_{\varepsilon}\psi\left(\frac{\Delta^{2}\psi}{\psi}-|x|^{\alpha}u^{p-1}\right)dx=-\int_{B_{1}}uH(\zeta_{\varepsilon},\psi)dx,\nonumber
\end{align}
where
\begin{align}
H(\zeta_{\varepsilon},\psi) = & 4\nabla\zeta_{\varepsilon}\cdot\nabla\Delta\psi+2\Delta\zeta_{\varepsilon}\Delta\psi+4\sum_{i,j=1}^{n}(\zeta_{\varepsilon})_{x_{i}x_{j}}\psi_{x_{i}x_{j}}\nonumber\\
& +4\nabla\Delta\zeta_{\varepsilon}\cdot\nabla\psi+\psi\Delta^{2}\zeta_{\varepsilon}.\nonumber
\end{align}
From Theorem \ref{thm1.1}, \eqref{step2.2} and \eqref{step2.3}, we get that
\begin{align}
\left|\int_{B_{1}}u H(\zeta_{\varepsilon},\psi)dx\right|&\leq \left|\int_{\{\frac{1}{2}\leq|x|\leq\frac{3}{4}\}}uH(\zeta_{\varepsilon},\psi)dx\right|+\left|\int_{\{\frac{\varepsilon}{2}\leq|x|\leq\varepsilon\}}uH(\zeta_{\varepsilon},\psi)dx\right|\nonumber\\
&\leq C_{1}+C_{2}\varepsilon^{-\gamma-4}\varepsilon^{n}\varepsilon^{-\frac{4+\alpha}{p-1}}\leq C<\infty, \nonumber
\end{align}
where $C>0$ is independent of $\varepsilon$.

On the other hand, the assumption \eqref{step1.1} implies that
$$u^{p-1}(x)=o(|x|^{-(4+\alpha)}) \quad\mbox{as}~ |x|\rightarrow0.$$
Therefore, we have
$$\int_{B_{1}}u\zeta_{\varepsilon}|x|^{-\gamma-4}dx \leq C <\infty.$$
Sending $\varepsilon\rightarrow0$, the claim is proved.

Now, we consider the two cases $-4< \alpha \leq 0$ and $0\leq \alpha <4$ separately.

{\bf Case 1:} $-4< \alpha \leq 0$. Note that $\frac{(p-1)n}{4+\alpha} >1$ due to $p>\frac{n+\alpha}{n-4}$ and $\alpha > -4$.  By Theorem \ref{thm1.1} and \eqref{Remty01} we have
\begin{equation}\label{Remty02}
\aligned
\int_{B_{1/2}} u^{\frac{(p-1)n}{4 + \alpha}} dx & \leq C \int_{B_{1/2}} u |x|^{-\frac{4+\alpha}{p-1} \big[ \frac{(p-1)n}{4+\alpha} -1 \big] } dx \\
& = C \int_{B_{1/2}}u |x|^{-n + \frac{4+\alpha}{p-1}} dx < \infty.
\endaligned
\end{equation}
Define the function $v(x)$ as follows
\begin{align}\label{step3.1}
v(x):=-u(x)+\int_{B_{1/2}}G(x,y) |y|^{\alpha}u^{p}(y)dy, \quad x\in B_{1/2},
\end{align}
where $G(x,y)$ is the Green's function of $\Delta^{2}$ in $B_{1/2}$ with homogenous Dirichlet boundary conditions. Then there exists a positive constant $C_{n}$ such that
$$0<G(x,y)\leq \Gamma(x,y):=C_{n}|x-y|^{4-n} \quad\mbox{for}~  x,y\in B_{1/2}, ~~x\neq y.$$
A similar argument as in \cite[Lemma 3.1]{Y2020} shows that $v$ satisfies
$$\Delta^{2}v=0  \quad \mbox{in}~ B_{1/2}$$
in the distributional sense. According to the interior regularity theory, we know that $v\in C(B_{1/2})$.

Set $a(x):=u^{p-1}(x)$. Then by \eqref{Remty02} we have $a \in L^{\frac{n}{4+\alpha}} (B_{1/2})$. For any large number $L>0$, let
\begin{eqnarray*}
a_{L}(x)=
\begin{cases}
a(x) &\mbox{if }|a(x)|\geq L ,\cr
0 &\mbox{otherwise},
\end{cases}
\end{eqnarray*}
and
$$a_{M}(x)=a(x)-a_{L}(x).$$
Then \eqref{step3.1} can be rewritten as
\begin{align}
u(x)&=\int_{B_{1/2}}G(x,y) |y|^{\alpha} u^{p}(y) dy - v(x)\nonumber\\
&=\int_{B_{1/4}}G(x,y) |y|^{\alpha} a_{L}(y) u(y)dy+\int_{B_{1/4}}G(x,y)|y|^{\alpha} a_{M}(y) u(y)dy\nonumber\\
& ~~~~~ +\int_{B_{1/2}\setminus B_{1/4}}G(x,y)|y|^{\alpha}u^{p}dy-v(x)\nonumber\\
&=: T_{L}u(x)+F_{L} u(x) + H(x) - v(x), \nonumber
\end{align}
where $T_L$ is a linear operator defined as
$$
T_{L}:w \rightarrow \int_{B_{1/4}} G(x,y) |y|^{\alpha} a_{L}(y) w(y) dy.
$$
Notice that
$$|H(x)|\leq C\int_{\{ \frac{1}{4}\leq|y|\leq\frac{1}{2}\} }G(x,y)dy\leq C\int_{B_{1}}|y|^{4-n}dy\leq C  \quad\mbox{for all}~ x\in B_{1/4}.$$
Hence $v,~H\in L^{\infty}(B_{1/4})$.

We shall prove that $x=0$ is a removable singularity. By Proposition \ref{pro7} it is sufficient to show that $T_{L}$ is a contracting operator from $L^{q}(B_{1/4})$ to $L^{q}(B_{1/4})$ for $L$ large and $F_{L}\in L^{q}(B_{1/4})$ for any $q \in (\frac{n}{n-4},  \infty)$.

For any $q \in (\frac{n}{n-4},  \infty)$, there exists $ r \in (\frac{n}{n+\alpha}, \frac{n}{4+\alpha}) $ such that (noticing that $\frac{n}{n+\alpha} \geq 1$ due to $-4< \alpha \leq 0$)
\begin{equation}\label{relqa}
\frac{1}{q}=\frac{1}{r}-\frac{4 + \alpha}{n}.
\end{equation}
By the doubly weighted Hardy-Littlewood-Sobolev inequality (see Proposition \ref{pro5}) and H\"{o}lder inequality, we obtain
\begin{align}
\| T_{L}w \|_{L^{q}(B_{1/4})} & \leq C \left\| \int_{B_{1/4}} |\cdot - y|^{4-n} |y|^{\alpha} a_{L}(y) w(y) dy \right\|_{L^{q}(B_{1/4})} \nonumber\\
& \leq C \|a_L w\|_{L^r(B_{1/4})} \nonumber\\
& \leq C\| a_{L} \|_{L^{\frac{n}{4 + \alpha}}(B_{1/4})} \| w \|_{L^{q}(B_{1/4})},\nonumber
\end{align}
where we used the fact $0\leq -\alpha < \frac{n}{r'} = n(1 - \frac{1}{r})$ because of $q > \frac{n}{n-4}$. Since $a \in L^{\frac{n}{4+\alpha}} (B_{1/2})$, we have for $L>0$ large enough that
$$
C\| a_{L} \|_{L^{\frac{n}{4 + \alpha}}(B_{1/4})} \leq \frac{1}{2}.
$$
Therefore, $T_{L}: L^{q}(B_{1/4})\rightarrow L^{q}(B_{1/4})$ is a contracting operator for large $L$.

On the other hand, by \eqref{Remty02} we know
$$
u\in L^{r}(B_{1/4}) \quad\mbox{for any}~ 1< r \leq \frac{(p-1)n}{4 + \alpha}.
$$
Note that $\frac{(p-1)n}{4 + \alpha} > \frac{n}{n+\alpha}$ due to $p > \frac{n+\alpha}{n-4}$ and $\alpha >-4$. Since $a_{M}$ is a bounded function, by the doubly weighted Hardy-Littlewood-Sobolev inequality we have
$$
\|  F_{L} u \|_{L^{q}(B_{1/4})}\leq \| a_{M}u \|_{L^{r}(B_{1/4})}\leq C \| u\|_{L^{r}(B_{1/4})},
$$
where $r$, $q$ satisfy \eqref{relqa} and $\frac{n}{n+\alpha} <  r \leq \frac{(p-1)n}{4 + \alpha}$.
It is easy to see that
$$
q=\frac{(p-1)n}{(4+\alpha)(2-p)} \quad\mbox{if}~ r=\frac{(p-1)n}{4 + \alpha}.
$$
This implies that $F_{L} u \in L^{q}(B_{1/4})$ for
$$
\begin{cases}
1<q <\infty  \quad & \mbox{if}~ p\geq2,\\
1<q \leq\frac{(p-1)n}{(4+\alpha)(2-p)}  \quad & \mbox{if}~ p<2.
\end{cases}
$$
Using the Regularity Lifting Theorem in Section \ref{sec2}, we obtain $u\in L^{q}(B_{1/4})$ for
$$
\begin{cases}
1<q<\infty  \quad & \mbox{if}~ p \geq 2,\\
1<q\leq\frac{(p-1)n}{(4+\alpha)(2-p)}  \quad & \mbox{if}~ p<2.
\end{cases}
$$
If $u\in L^{\frac{(p-1)n}{(4+\alpha)(2-p)}}(B_{1/4})$ (for $p<2$), then we have $F_{L} u \in L^{q}(B_{1/4})$ for
$$
\begin{cases}
1< q < \infty  \quad & \mbox{if}~p\geq\frac{3}{2}, \\
1< q \leq \frac{(p-1)n}{(4+\alpha)(3-2p)}  \quad & \mbox{if}~ p<\frac{3}{2}.
\end{cases}
$$
Using the Regularity Lifting Theorem again, we obtain $u\in L^{q}(B_{1/4})$ for
$$
\begin{cases}
1< q < \infty  \quad & \mbox{if}~p\geq\frac{3}{2}, \\
1< q \leq \frac{(p-1)n}{(4+\alpha)(3-2p)}  \quad & \mbox{if}~ p<\frac{3}{2}.
\end{cases}
$$
Continuing this process, a finite number of iterations gives
$$u\in L^{q}(B_{1/4}) \quad\mbox{for any}~ 1<q<\infty.$$
Therefore, we have by H\"{o}lder inequality,
$$\int_{B_{1/4}}G(x,y)|y|^{\alpha}u^{p}(y)dy\in L^{\infty}(B_{1/4}).$$
This implies $u\in L^{\infty}(B_{1/4})$ by using the integral equation \eqref{step3.1}. Combined with the fact $|\cdot|^\alpha \in L^{q_0}(B_{1/2})$ for some $q_0 > n/4$, we easily obtain that $x=0$ is a removable singularity.

Next, we discuss the case $0\leq \alpha <4$. Since the method is similar to the above, we only give the details that need to be changed.

{\bf Case 2:} $0\leq \alpha <4$. In this case, we have $(p-1)\frac{n}{4} - 1 >0$. By Theorem \ref{thm1.1} and \eqref{Remty01} we obtain
\begin{equation}\label{Remty03}
\aligned
\int_{B_{1/2}} (|x|^{\alpha}u^{p-1})^{\frac{n}{4}} dx & \leq C\int_{B_{1/2}}u |x|^{\frac{\alpha n}{4} - \frac{4+\alpha}{p-1} \big[ (p-1)\frac{n}{4} - 1\big]} dx  \\
& = C \int_{B_{1/2}}u |x|^{-n + \frac{4+\alpha}{p-1}} dx < \infty.
\endaligned
\end{equation}
Set $b(x):=|x|^\alpha u^{p-1}(x)$. Then by \eqref{Remty03} we have $b \in L^{\frac{n}{4}} (B_{1/2})$. For any large number $L>0$, let
\begin{eqnarray*}
b_{L}(x)=
\begin{cases}
b(x) &\mbox{if }|b(x)|\geq L ,\cr
0 &\mbox{otherwise},
\end{cases}
\end{eqnarray*}
and
$$b_{M}(x)=b(x)-b_{L}(x).$$
Similar to Case 1, we can obtain the following integral equation
$$
u(x)=T_{L}^b u(x)+F_{L}^b u(x) + H(x) - v(x),
$$
where $v,~H\in L^{\infty}(B_{1/4})$, $T_{L}^b$ is a linear operator given by
$$
T_{L}^b w(x) = \int_{B_{1/4}} G(x,y) b_{L}(y) w(y) dy,
$$
and
$$
F_{L}^b u(x) =  \int_{B_{1/4}} G(x,y) b_{M}(y) u(y)dy.
$$
For any $q \in (\frac{n}{n-4},  \infty)$, there exists $ r \in (1, \frac{n}{4}) $ such that
\begin{equation}\label{relqa021}
\frac{1}{q}=\frac{1}{r}-\frac{4}{n}.
\end{equation}
By Hardy-Littlewood-Sobolev inequality and H\"{o}lder inequality, we obtain that $T_{L}^b: L^{q}(B_{1/4})\rightarrow L^{q}(B_{1/4})$ is a contracting operator for large $L$.

Moreover, it is easy to verify that \eqref{Remty02} still holds in this case. Namely,
$$
u\in L^{r}(B_{1/4}) \quad\mbox{for any}~ 1< r \leq \frac{(p-1)n}{4 + \alpha}.
$$
By the boundedness of $b_{M}$ and Hardy-Littlewood-Sobolev inequality, we have $F_{L}^b u \in L^{q}(B_{1/4})$ for
$$
\begin{cases}
1<q <\infty  \quad & \mbox{if}~ p \geq (8+\alpha)/4, \\
1<q \leq\frac{(p-1)n}{8+\alpha-4p}  \quad & \mbox{if}~ p< (8+\alpha)/4.
\end{cases}
$$
Using the Regularity Lifting Theorem, we obtain $u\in L^{q}(B_{1/4})$ for
$$
\begin{cases}
1<q <\infty  \quad & \mbox{if}~ p \geq (8+\alpha)/4, \\
1<q \leq\frac{(p-1)n}{8+\alpha-4p}  \quad & \mbox{if}~ p< (8+\alpha)/4.
\end{cases}
$$
Similar to Case 1, a finite number of iterations yields
$$u\in L^{q}(B_{1/4}) \quad\mbox{for any}~ 1<q<\infty.$$
The rest is the same as Case 1, so we omit the details.  The proof of Proposition \ref{Remov} is completed.
\end{proof}

\noindent\textbf{Proof of Theorem \ref{thm1.3}} It follows from Theorem \ref{thm1.1}, Propositions \ref{lim} and \ref{Remov}.

\medskip

\noindent\textbf{Proof of Theorem \ref{mth1}.} Our proof is based on the monotonicity formula in Proposition \ref{Mon01}, the radial symmetry in Proposition \ref{pro2} and the removable singularity in Proposition \ref{Remov}.

Let $u\in C^{4}(B_{1}(0)\setminus\{0\})$ be a nonnegative solution of \eqref{eq1} with $m=2$.  We claim that either
\begin{align}\label{step1.1-01}
\lim_{|x|\rightarrow0}|x|^{\frac{4+\alpha}{p-1}}u(x)=0
\end{align}
or
\begin{align}\label{step1.2}
\lim_{|x|\rightarrow0}|x|^{\frac{4+\alpha}{p-1}}u(x)=A_{0}^{\frac{1}{p-1}}.
\end{align}
For $\lambda>0$, define $u^{\lambda}(x)=\lambda^{\frac{4+\alpha}{p-1}}u(\lambda x)$. It follows from Theorem \ref{thm1.1} that $u^{\lambda}$ is  uniformly bounded in $C^{4,\gamma}(K)$ on every compact set $K\subset B_{1/{2\lambda}} \setminus\{0\}$ for some $0<\gamma<1$. Therefore, there exists a nonnegative function $u^{0}\in C^{4}(\mathbb{R}^{n}\setminus\{0\})$ such that,  after passing to a subsequence of $\lambda\rightarrow0$,
$$u^{\lambda} \rightarrow u^{0} \quad\mbox{in}~ C_{loc}^{4}(\mathbb{R}^{n}\setminus\{0\}) \quad  \mbox{as}~ \lambda\rightarrow0,$$
and $u^0$ satisfies
$$\Delta^{2}u^{0}=|x|^{\alpha}(u^{0})^{p} \quad\mbox{in} ~\mathbb{R}^{n}\setminus\{0\}.$$
It follows from Proposition \ref{pro2} that $u^{0}$ is radially symmetric with respect to the origin. Moreover, by the scaling invariance of $\widetilde{E}$, we have that for any $r>0$
$$\widetilde{E}(r;u^{0})=\lim_{\lambda\rightarrow0}\widetilde{E}(r;u^{\lambda})=\lim_{\lambda\rightarrow0}\widetilde{E}(r\lambda;u)=\widetilde{E}(0;u).$$
This implies that $\widetilde{E}(r;u^{0})$ is a constant. Let
$$w^{0}(t)=|x|^{\frac{4+\alpha}{p-1}}u^{0}(|x|) \quad \mbox{with} ~ t=\ln|x|.$$
Then the function $w^{0}$ satisfies
\begin{equation}\label{w}
\partial^{(4)}_{t}w^{0}+A_{3}\partial^{(3)}_{t}w^{0}+A_{2}\partial_{tt}w^{0}+A_{1}\partial_{t}w^{0}+A_{0}w^{0}=(w^{0})^{p}
\end{equation}
and
$$E(t;w^{0})=\widetilde{E}(r;u^{0}) \equiv const. $$
Thus, we get
$$0=\frac{d}{dt}E(t;w^{0})=|\mathbb{S}^{n-1}|\left[A_{3}(\partial_{tt}w^{0})^{2}-A_{1}(\partial_{t}w^{0})^{2}\right].$$
Since $A_{3}<0$, $A_{1}>0$ (or $A_3>0,~A_1<0$), we obtain that $w^{0}$ is a constant. Furthermore, by the equation \eqref{w} we have
$$\mbox{either}\quad w^{0}\equiv0 \quad\mbox{or}\quad w^{0}=A_{0}^{\frac{1}{p-1}}.$$
This, combined with the existence of $\lim_{r \rightarrow 0^+}\widetilde{E}(r;u)$, we obtain that either
$$\lim_{|x|\rightarrow0}|x|^{\frac{4+\alpha}{p-1}}u(x)=0$$
or
$$\lim_{|x|\rightarrow0}|x|^{\frac{4+\alpha}{p-1}}u(x)=A_{0}^{\frac{1}{p-1}}.$$
Together with Proposition \ref{Remov}, we proved Theorem \ref{mth1}.
$\hfill\square$

\medskip
\noindent{\bf Acknowledgements.}  X. Huang is partially supported by NSFC (No.~12271164); the Shanghai Frontier Research Center of Modern Analysis. Huang and Li are also supported in part by Science and Technology Commission of Shanghai Municipality (No.~22DZ2229014).

\bigskip

\noindent X. Huang \\
\noindent School of Mathematical Sciences and Shanghai Key Laboratory of PMMP \\
East China Normal University, Shanghai 200241, China \\
Email: \textsf{xhuang@cpde.ecnu.edu.cn}

\bigskip

\noindent Y. Li \\
\noindent School of Mathematical Sciences and Shanghai Key Laboratory of PMMP \\
East China Normal University, Shanghai 200241, China \\
Email: \textsf{yli@math.ecnu.edu.cn}

\bigskip

\noindent H. Yang \\
\noindent Department of Mathematics, The Hong Kong University of Science and Technology\\
Clear Water Bay, Kowloon, Hong Kong, China \\
Email: \textsf{mahuiyang@ust.hk}

\end{document}